\documentclass[11pt]{amsart}
\usepackage[textheight = 615pt,textwidth = 360pt]{geometry}
\usepackage{mathrsfs}
\usepackage{color}
\usepackage{bm}
\usepackage{amsfonts,verbatim,amssymb}
\usepackage{dsfont}
\usepackage{amscd}
\usepackage{extarrows}
\usepackage{amsmath}
\usepackage{mathrsfs}
\usepackage{enumerate}
\usepackage{amscd}
\usepackage[all]{xy}
\usepackage{hyperref}
\numberwithin{equation}{section}
\theoremstyle{plain} 
\newtheorem{theorem}{\indent\bf Theorem}[section]

\theoremstyle{definition} 

\newtheorem{remark}[theorem]{\bf Remark}
\newtheorem{thm}{Theorem}[section]
\newtheorem{cor}[thm]{Corollary}
\newtheorem{lem}[thm]{Lemma}

\theoremstyle{definition}

\theoremstyle{remark}

\newcommand{\bea}{\begin{eqnarray}}
	\newcommand{\eea}{\end{eqnarray}}
\newcommand{\ben}{\begin{eqnarray*}}
	\newcommand{\een}{\end{eqnarray*}}
\newcommand{\bt}{\begin{split}}
	\newcommand{\et}{\end{split}}

	\newcommand{\mc}{\mathbb{C}}
	\newcommand{\mr}{\mathbb{R}}

	\newcommand{\beq}{\begin{equation*}}
		\newcommand{\eeq}{\end{equation*}}
	\newcommand{\bi}{\begin{itemize}}
		\newcommand{\ei}{\end{itemize}}

\newcommand{\mo}{\mathcal{O}}

\newcommand{\rw}{\rightarrow}
\newcommand{\rwo}{\mapsto}

\newcommand{\re}{\operatorname{Re}}
\newcommand{\im}{\operatorname{Im}}

\DeclareMathOperator{\tr}{Tr}
	
\DeclareMathOperator{\dom}{Dom}
\allowdisplaybreaks[4]
	
\begin{document}

\title[{$\bar\partial$} Sobolev-type inequality and an improved $L^2$-estimate]
{{$\bar\partial$} Sobolev-type inequality and an improved $L^2$-estimate of $\bar\partial$ on bounded strictly pseudoconvex domains}

\author[F. Deng]{Fusheng Deng}
\address{Fusheng Deng: \ School of Mathematical Sciences, University of Chinese Academy of Sciences \\ Beijing 100049, P. R. China}
\email{fshdeng@ucas.ac.cn}

\author[W. Jiang]{Weiwen Jiang}
\address{Weiwen Jiang: \ Institute of Mathematics, Academy of Mathematics and Systems Science, Chinese Academy of Sciences, Beijing 100190, P. R. China}
\email{jiangweiwen@amss.ac.cn}

\author[X. Qin]{Xiangsen Qin}
\address{Xiangsen Qin: \ School of Mathematical Sciences, University of Chinese Academy of Sciences \\ Beijing 100049, P. R. China}
\email{qinxiangsen19@mails.ucas.ac.cn}

\begin{abstract}
We prove several Sobolev-type inequalities related to the $\bar\partial$-operator on bounded domains in $\mc^n$, which can be viewed as
a $\bar\partial$-version of the classical Sobolev inequality and its various generalizations, and apply them to derive a generalization of the Sobolev Inequality with Trace in $\mr^n$.
As applications to complex analysis, we get an integral form of  Maximum Modulus Principle for holomorphic functions, and an improvement of H\"ormander's $L^2$-estimate for $\bar\partial$ on bounded strictly pseudoconvex domains.
\end{abstract}

\maketitle
\tableofcontents
    \section{Introduction}
    \indent For any Lebesgue measurable function $f$ on an open subset $D\subset\mr^n$ and any $1\leq p\leq \infty$, let $\|f\|_{L^p(D)}$
    denote its $L^p$-norm. Similar for $\|\cdot\|_{L^p(\partial D)}$ if $D$ has Lipschitz boundary.\\
    \indent The following Poincar\'e Inequality is well known. Itself and its various modifications are of great interest in partial differential equations
    and isoperimetry theory (see e.g. \cite{Evans}, \cite{Gilbarg} and \cite{maggi}).
    \begin{thm}[{$L^2$} Poincar\'e Inequality]\label{thm:real poincare inequality}
    Let $D\subset\mr^n$ be a bounded domain with Lispchitz boundary, and let $R$ be the diameter of $D$, then
    there is a constant $\delta:=\delta(R)>0$ such that
    $$
    \delta\|f\|_{L^2(D)}\leq\|\nabla f\|_{L^2(D)}+\|f\|_{L^2(\partial D)},\ \forall f\in C^1(\overline D).
    $$
    \end{thm}
    In the above theorem, $\nabla f$ is the gradient of $f$, $C^1(\overline D)$ represents the space of  complex valued functions which are of class $C^1$ in an open neighborhood of the closure $\overline D$ of $D$.\\
    \indent Theorem \ref{thm:real poincare inequality} has a more general version as follows.
    \begin{thm}\label{thm:real sobolev inequality}
    Let $D\subset\mr^n$ be a bounded domain with Lipschitz boundary, then
    \begin{itemize}
      \item[(i)] If $p,q,r$ satisfy
    $$1\leq p<n,\ 1\leq q<\frac{np}{n-p},\ \frac{(n-1)q}{n}<r<\infty,$$
    then  there is a constant $\delta:=\delta(|D|,n,p,q,r)>0$ such that
    $$
      \delta \|f\|_{L^q(D)}\leq \|\nabla f\|_{L^p(D)}+\|f\|_{L^r(\partial D)},\ \forall f\in C^1(\overline D),
    $$
    where $|D|$ is the Lebesgue measure of $D$.
      \item[(ii)] If $p,q$ satisfy
      $$2\leq n< p<\infty,\ 1\leq q\leq \infty, $$
  then there is a constant $\delta:=\delta(n,p,q)>0$ such that for any compact supported $f\in C^1(D)$,we have
    $$
    \delta\|f\|_{L^q(D)}\leq |D|^{\frac{1}{n}+\frac{1}{q}-\frac{1}{p}}\|\nabla f\|_{L^p(D)}.
    $$
    \end{itemize}
    \end{thm}
    Part (i) (which is called Sobolev Inequality with Trace) is a slight modification (using H\"older's Inequality) of \cite[Theorem 1.2]{maggi} (see also \cite[Formula (2.1)]{Anv}, and part (ii)
   (which is called Morrey's Inequality) is a slight modification of \cite[Theorem 7.10]{Gilbarg}.
    However, the case $p=n$ is more subtle, see \cite[Theorem 7.15]{Gilbarg}. We do not find a version of part (ii) of Theorem \ref{thm:real sobolev inequality} for functions without compact support in the literature. In the present paper, we derive a generalized form  of Theorem \ref{thm:real sobolev inequality} (without giving the explicit form of the constant in the estimates) from analogous estimates for the $\bar\partial$-operator, see Corollary \ref{thm: general sobolev inequality}.\\
    \indent  From the viewpoint of complex analysis, it is very natural to ask whether similar estimates still hold for the $\bar\partial$-operator,
    which is the main topic of the present paper. Our initial motivation to this work is to give an improvement of H\"ormander's $L^2$-estimate of $\bar\partial$ on bounded strictly pseudoconvex domains.
    
    The first main result is the following
    \begin{thm}\label{thm:L^p poincare inequality}
    Let $\Omega\subset\mc^n$ be a bounded domain with Lipschitz boundary, and let $1\leq p<\infty$, then there is a constant $\delta:=\delta(\Omega,n,p)>0$ such that
    $$
    \delta\|f\|_{L^p(\Omega)}\leq \|\bar\partial f\|_{L^p(\Omega)}+\|f\|_{L^p(\partial\Omega)},
    $$
    for all $f\in C^1(\Omega)\cap C^0(\overline\Omega)$ such that $\bar\partial f$ has a continuous continuation to $\overline\Omega$.
    \end{thm}
    In the above theorem, $C^1(\Omega)$ denotes the space of complex valued functions which are of class $C^1$ on $\Omega$, $C^0(\overline{\Omega})$ represents the space of  complex valued continuous functions on $\overline{\Omega}$, and we identify $\bar{\partial}f$ with $\left(\frac{\partial f}{\partial \bar{z}_1},\cdots,\frac{\partial f}{\partial\bar{z}_n}\right)$.\\
    \indent For the case $p=2$, Theorem \ref{thm:L^p poincare inequality} is a $\bar\partial$-version of the Poincar\'e inequality presented
    as in Theorem \ref{thm:real poincare inequality}. If $p=2$ and $f$ is real valued, Theorem \ref{thm:L^p poincare inequality} is a consequence of Theorem \ref{thm:real poincare inequality}
    since in this case $|\nabla f|=2|\bar\partial f|.$ If $p=2$ and $f$ has compact support,  Theorem \ref{thm:L^p poincare inequality} is also a consequence of Theorem \ref{thm:real poincare inequality}
    as
    $$
    \int_{\Omega}\left|\frac{\partial f}{\partial z_j}\right|^2d\lambda=\int_{\Omega}\left|\frac{\partial f}{\partial \bar{z}_j}\right|^2d\lambda,
    \ \forall\ 1\leq j\leq n.
    $$
    When $\Omega$ is strictly pseudoconvex with $C^2$-boundary and $p=2$, the result of Theorem \ref{thm:L^p poincare inequality} is already known (e.g. see the proof of \cite[Proposition 2.2.2]{Ber95}). The general case is beyond the scope of Theorem \ref{thm:real poincare inequality} since one cannot control the norm of $\bar\partial f$ in term of that of
    $\nabla f$, and beyond the discussion in \cite{Ber95} since the strict pseudoconvexity condition is indispensable for the arguments there. \\
  \indent  We don't know whether Theorem \ref{thm:L^p poincare inequality} still holds or not if $p=\infty$.\\
    \indent Parallel to Theorem \ref{thm:real sobolev inequality}, we also consider a generalization of Theorem \ref{thm:L^p poincare inequality} as follows.
    \begin{thm}\label{thm:sobolev inequality}
        Let $\Omega\subset\mc^n$ be a bounded domain with Lipschitz boundary
        and let $p,q,r$ satisfy
        $$
         1\leq p<\infty,\ 1\leq q<\infty,\ 1\leq r<\infty,\ q(2n-p)<2np,\ q(2n-1)<2nr,
        $$
        then there is a constant $\delta:=\delta(\Omega,p,q,r)>0$ such that
    $$
    \delta\|f\|_{L^q(\Omega)}\leq \|\bar\partial f\|_{L^p(\Omega)}+\|f\|_{L^r(\partial\Omega)},
    $$
    for all $f\in C^1(\Omega)\cap C^0(\overline\Omega)$ such that $\bar\partial f$ has a continuous continuation to $\overline\Omega.$
   \end{thm}
   Compared to part (i) of Theorem \ref{thm:real sobolev inequality}, we don't need to assume that $p<2n$ in Theorem
   \ref{thm:sobolev inequality}.\\
   \indent  The constant $\delta$ in the above theorem can be written explicitly by tracking our proof,  and all constants in the following theorems and corollaries can be chosen similarly.
   For the sake of simplicity, we will not repeat this argument.
   \begin{remark}\label{remark:sobolev}
    Since $C^1(\overline\Omega)$ is dense in the Sobolev space $W^{1,p}(\Omega)$ for any $1\leq p<\infty$,
   see \cite[Theorem 3, Subsection 5.3, Page 252]{Evans}, and the trace operator $\tr\colon W^{1,r}(\Omega)\rw L^r(\partial\Omega)$
   is continuous, see \cite[Theorem 1, Subsection 5.5, Page 258]{Evans}, then Theorem \ref{thm:sobolev inequality} can be easily
   generalized to functions in the Sobolev space $W^{1,p}(\Omega)\cap W^{1,r}(\Omega)$ by tracking our proof.
   \end{remark}
   \indent  Now we discuss the main idea of the proof of Theorem \ref{thm:sobolev inequality}. Although Theorem \ref{thm:sobolev inequality}
    looks similar to Theorem \ref{thm:real sobolev inequality} in form, it seems difficult to modify the idea in the proofs of Theorem \ref{thm:real poincare inequality}
     and Theorem \ref{thm:real sobolev inequality} in the literature to give a proof of Theorem \ref{thm:sobolev inequality}. The key ingredients in our arguments for Theorem \ref{thm:sobolev inequality}
    are Bochner-Martinelli Formula and H\"older's Inequality for three functions.\\
    \indent Based on  a principle that connects convex analysis and complex analysis (proposed in \cite{DHJQ}), we derive the following result from Theorem \ref{thm:sobolev inequality}.
     \begin{cor}\label{thm: general sobolev inequality}
        Let $D\subset\mr^n$ be a bounded domain with Lipschitz boundary and let $p,q,r$ satisfy
        $$
         1\leq p<\infty,\ 1\leq q<\infty,\ 1\leq r< \infty,\ q(2n-p)<2np,\ q(2n-1)<2nr,
        $$
        then there is a constant $\delta:=\delta(D,p,q,r)>0$ such that
    $$
    \delta\|f\|_{L^q(D)}\leq \|\nabla f\|_{L^p(D)}+\|f\|_{L^r(\partial D)},\ \forall f\in C^1(\overline D).
    $$
   \end{cor}
   \indent We now discuss the relationships between Theorem \ref{thm:sobolev inequality}, Corollary \ref{thm: general sobolev inequality} and Theorem \ref{thm:real sobolev inequality}.
   Clearly, for domains $D\subset\mc^n\cong \mr^{2n}$, Theorem \ref{thm:sobolev inequality} can be seen as a generalization of Theorem \ref{thm:real sobolev inequality}
   (overlooking the estimates of the constants in the inequalities). Corollary \ref{thm: general sobolev inequality} is stronger than Theorem \ref{thm:real sobolev inequality}
    in the sense that $p$ can be larger than $n$ and $f$ can be taken to have noncompact support,
    and a little bit weaker than Theorem \ref{thm:real sobolev inequality} in the sense that the condition in Corollary \ref{thm: general sobolev inequality} may lead to a narrower range of $p,q,r$. \\
    \indent  In connection to Remark \ref{remark:sobolev}, Corollary \ref{thm: general sobolev inequality} can be generalized to Sobolev spaces
    accordingly.\\
    \indent For any domain $D\subset\mr^n$, let $d\lambda$ be the Lebesgue measure on $D$,  and let $dS$ be the surface measure on the boundary $\partial D$ of $D$ if $D$ has 
    Lipschitz boundary. For application to $L^2$-estimate of $\bar\partial$ on strictly pseudoconvex domains, we need certain weighted version of Theorem \ref{thm:sobolev inequality} as follows.
    \begin{thm}\label{thm:poincare inequality with weight}
        Let $\Omega\subset\mc^n$ be a bounded domain with a $C^2$-boundary defining function $\rho$.
        Set 
        $$K:=\frac{\sup_{z\in \Omega}|\nabla\rho(z)|+\sup_{z\in \Omega}|\Delta\rho(z)|}{\inf_{x\in\partial \Omega}|\nabla\rho(z)|},$$
         where $\Delta\rho$ is the Laplacian of $\rho$.
        Assume $\varphi\in C^0(\overline\Omega),$  and set 
        $$m_1:=\inf_{z\in \Omega}e^{-\varphi(x)},\ m_2:=\sup_{z\in \Omega}e^{-\varphi(z)},$$
        then there is a constant $\delta:=\delta(|\Omega|,|\partial\Omega|,K,m_1,m_2,n)>0$ such that
        $$
        \delta \int_{\Omega}|f|^2e^{-\varphi}d\lambda\leq \int_{\Omega}|\bar\partial f|^2e^{-\varphi}d\lambda+\int_{\partial\Omega}|f|^2e^{-\varphi}d\lambda,
        \ \forall f\in C^1(\overline\Omega).
        $$
   \end{thm}
    Theorem \ref{thm:poincare inequality with weight} is a direct consequence of Theorem \ref{thm:L^p poincare inequality} as the constant can be tracked, so we omit it here.\\
    \indent We now consider applications of Theorem \ref{thm:sobolev inequality} and Theorem \ref{thm:poincare inequality with weight} to complex analysis.
    Firstly, a direct interesting consequence of Theorem \ref{thm:sobolev inequality} is the following Corollary \ref{cor:maximum modulus}, which
    can be seen as an integral form of the Maximum Modulus Principle for holomorphic functions.
    \begin{cor}\label{cor:maximum modulus}
        Let $\Omega\subset\mc^n$ be a bounded domain with Lipschitz boundary and let $p,q$ satisfy
        $$
        1\leq p<\infty,\ 1\leq q<\infty,\ q(2n-1)<2np,
        $$
         then there is a constant $\delta:=\delta(\Omega,p,q)>0$ such that
        $$
        \delta\int_{\Omega}|f|^qd\lambda\leq \int_{\partial\Omega}|f|^pdS,\ \forall f\in \mo(\Omega)\cap C^0(\overline{\Omega}),
        $$
        where $\mo(\Omega)$ denotes the space of holomorphic functions on $\Omega$.
    \end{cor}
    \indent Secondly, we derive an improvement of H\"ormander's $L^2$-estimate of $\bar\partial$ (see \cite[Theorem 1.6.4]{Ber95} for an appropriate formulation) when $\Omega$ is strictly pseudoconvex from Theorem \ref{thm:poincare inequality with weight}. Note that in Theorem 1.6.4 of \cite{Ber95}, the right hand side of Inequality (\ref{formula:L2 estimate}) is 
    $$\int_{\Omega}\sum_{j,k=1}^n \psi^{j\bar k}f_j\bar{f}_ke^{-\varphi-\psi}d\lambda.$$
    \begin{thm}\label{thm:strictly pseudoconvex}
        Let $\Omega\subset \mc^n$ be a bounded strictly pseudoconvex domain with a $C^2$-boundary defining function
        $\rho,$ and let $c_0$ be the smallest eigenvalue of the complex Hessian $\left(\frac{\partial^2\rho}{\partial z_j\partial\bar{z}_k}\right)$ of $\rho$ on $\partial\Omega$.
        Let $\varphi\in C^2(\overline\Omega)$ be a plurisubharmonic function, and let $\psi\in C^2(\overline\Omega)$ be a strictly plurisubharmonic function.
        Let
        $$K:=\frac{\sup_{z\in\Omega}|\nabla \rho(z)|+\sup_{z\in\Omega}|\Delta\rho(z)|}{\inf_{z\in\partial\Omega}|\nabla \rho(z)|},$$
        and set
        $$m_1:=\inf_{z\in\Omega}e^{-\varphi(z)-\psi(z)},\ m_2:=\sup_{z\in\Omega}e^{-\varphi(z)-\psi(z)},
        \ m_3:=\inf_{z\in \partial\Omega}\frac{c_0(z)}{|\nabla\rho(z)|}.$$
        Then there is a constant $\delta:=\delta(|\Omega|,|\partial\Omega|,m_1,m_2,m_3,n)>0$ such that for any nonzero $\bar{\partial}$-closed Lebesgue measurable $(0,1)$-form $f:=\sum_{j=1}^n f_jd\bar{z}_j$
         satisfying
        $$
        M_f:=\int_{\Omega}\sum_{j,k=1}^n \psi^{j\bar k}f_j\bar{f}_ke^{-\varphi-\psi}d\lambda<\infty,
        $$
        we can solve $\bar{\partial}u=f$ with the estimate
        \begin{equation}\label{formula:L2 estimate}
        \int_{\Omega}|u|^2e^{-\varphi-\psi}d\lambda\leq\frac{\|f\|_{2}}{\sqrt{\|f\|_{2}^2+\delta M_f}}\int_{\Omega}\sum_{j,k=1}^n \psi^{j\bar k}f_j\bar{f}_ke^{-\varphi-\psi}d\lambda,
        \end{equation}
        where $(\psi^{j\bar k})_{1\leq j,k\leq n}$ is the inverse of the complex Hessian of $\psi$,
        $$
        \|f\|_{2}^2:=\sum_{j=1}^n\int_{\Omega}|f_j|^2e^{-\varphi-\psi}d\lambda.$$
    \end{thm}
    \indent  The main idea of the proof of Theorem \ref{thm:strictly pseudoconvex} can be summarized as follows.
     Based on Theorem \ref{thm:poincare inequality with weight}, we apply the last two terms on the right hand side of the Kohn-Morrey-H\"ormander formula (see
     Lemma \ref{basic identity}), which is  thrown away in the classical method to $L^2$-estimate of $\bar\partial$, to deduce a better
      estimate when $\Omega$ is strictly pseudoconvex.\\
     \indent Theorem \ref{thm:strictly pseudoconvex} can be easily generalized to any $(p,q)$-forms and any bounded $q$-pseudoconvex domains. But we don't pursue this here.\\
    \indent If we take $\psi(z)=\frac{|z-z_0|^2}{R^2}$ for some $z_0\in\Omega$, where $R$ is the diameter of $\Omega$, then Theorem \ref{thm:strictly pseudoconvex} gives the following
    \begin{cor}\label{cor:strictly pseudoconvex1}
    Let $\Omega\subset \mc^n$  be a bounded strictly pseudoconvex domain with a $C^2$-boundary defining function
        $\rho,$ and let $c_0$ be the smallest eigenvalue of the complex Hessian $\left(\frac{\partial^2\rho}{\partial z_j\partial\bar{z}_k}\right)$ of $\rho$ on $\partial\Omega$.
        Let $R$ be the diameter of $\Omega$, 
        $$K:=\frac{\sup_{z\in\Omega}|\nabla \rho(z)|+\sup_{z\in\Omega}|\Delta\rho(z)|}{\inf_{z\in\partial\Omega}|\nabla \rho(z)|},$$
        and let $\varphi\in C^2(\overline\Omega)$ be a plurisubharmonic function.  Set 
        $$m_1:=\inf_{z\in\Omega}e^{-\varphi(z)},\ m_2:=\sup_{z\in\Omega}e^{-\varphi(z)},\ m_3:=\inf_{z\in \partial\Omega}\frac{c_0(z)}{|\nabla\rho(z)|}$$ Then there exists a constant $\delta:=\delta(|\Omega|,|\partial\Omega|,K,m_1,m_2,m_3,n)>0$  such that for any nonzero $\bar{\partial}$-closed Lebesgue measurable $(0,1)$-form $f$ satisfying   $$
        \int_{\Omega}|f|^2e^{-\varphi}d\lambda<\infty,
        $$
        we can solve $\bar{\partial}u=f$ with the estimate
        \begin{equation}\label{inequality:2}
        \int_{\Omega}|u|^2e^{-\varphi}d\lambda\leq\frac{eR^2}{\sqrt{1+\delta}}\int_{\Omega}|f|^2e^{-\varphi}d\lambda.
        \end{equation}
    \end{cor}
    \indent Corollary \ref{cor:strictly pseudoconvex1} is an improvement of \cite[Theorem 2.2.3]{Hor65}
    in the sense that the constant $eR^2$ is improved to $\frac{eR^2}{\sqrt{1+\delta}}$ in Inequality (\ref{inequality:2})  if $\Omega$ is strictly pseudoconvex and $\varphi\in C^2(\overline\Omega)$.
    The constant $\delta$  can be also explicitly computed. One may wonder whether the constant $\delta_0$ in Corollary \ref{cor:strictly pseudoconvex1} can be chosen to be independent of $\varphi$.
    But this is impossible by the main result in \cite{DNW}.\\
    \indent It is of great importance and interest to consider optimal constants in the estimates presented in all above theorems and corollaries,
    but we do not discuss this topic in depth in the present work.
    It is possible to generalize Theorem \ref{thm:L^p poincare inequality} and Theorem \ref{thm:sobolev inequality} to the Laplace operator and other differential operators on 
    more general spaces, which will be systematically studied in forthcoming works.
    
   \subsection*{Acknowledgements}
    		The first author is grateful to Professor Xiangyu Zhou and Yuan Zhou for valuable discussions on related topics.
    		This research is supported by National Key R\&D Program of China (No. 2021YFA1003100),
            NSFC grants (No. 11871451, 12071310), and the Fundamental Research Funds for the Central Universities.
    \section{Notations and Conventions}\label{sec:notations}
    In this section, we fix some notations and conventions. \\
    \indent Our convention for $\mathbb{N}$ is that $\mathbb{N}:=\{1,2,3,\cdots\}$, we also set $\mc^*:=\mc\setminus\{0\}.$ For any complex number $z\in \mc$,
    let $\re(z)$ denote its real part and $\im(z)$ denote its imaginary part. For any $z\in\mc^n$, we always write $z=(z_1,\cdots,z_n)$.
    Similarly, for any $x\in\mr^n$, we always write $x=(x_1,\cdots,x_n)$. Let $\omega_n$ be the surface area of the unit sphere $\mathbb{S}^n:=\{x\in\mr^n|\ |x|=1\}$. \\
    \indent We say $D\subset\mr^n$ is a domain if $D$ is a connected open subset of $\mr^n$. For any domain $D\subset\mr^n$, let $\overline{D}$ denote the closure of $D$ in $\mr^n$,
     and let $\partial D:=\overline{D}\setminus D$ be its boundary. For any $k\in\mathbb{N}\cup \{0\}$, let $C^k(D)$ be the space of complex valued functions on
     $D$ which are of class $C^k$, and let $C^k(\overline D)$ denote the space of complex valued functions which are of class $C^k$ in an open neighborhood of $\overline D$.
     We also set
     $$C^k(\overline D;\mr^n):=\{F=(f_1,\cdots,f_n)\colon \overline D\rw \mr^n|\ f_j\in C^k(\overline D),\ \forall j=1,\cdots,n\}.$$
     For any $f\in C^1(\overline D)$, let $\nabla f$ denote its gradient. For any $f\in C^2(D)$, let $\Delta f$ denote its Laplacian.\\
     \indent Let $d\lambda$ denote the Lebesgue measure on $\mr^n$, and let $|A|$ denote the measure of a Lebesgue measurable subset $A\subset\mr^n$.  If a domain $D\subset \mr^n$ has Lipschitz boundary, let $dS$ denote the ($(n-1)$-)dimensional Hausdorff measure defined on $\partial D$, and we also let $|\partial D|$ be the Hausdorff measure of $\partial D$.\\
     \indent Let $\Omega\subset\mc^n$ be a domain, let $\mo(\Omega)$ denote the space of holomorphic functions on $\Omega$.
     We recall a real valued function $\varphi\in C^2(\Omega)$ is plurisubharmonic (resp. strictly plurisubharmonic) if its complex Hessian $\left(\varphi_{j\bar{k}(z)}\right)_{1\leq j,k\leq n}$ is positive semi-definite (resp. positive definite) for any $z\in \Omega$, where
     $$
     \varphi_{j\bar{k}}(z):=\frac{\partial^2\varphi}{\partial z_j\partial\bar z_k}(z),\ 1\leq j,k\leq n,
     $$
     and we let $\left(\varphi^{j\bar{k}}\right)_{1\leq j,k\leq n}$ be the inverse matrix of its complex Hessian if $\varphi$ is strictly plurisubharmonic.
     We also recall that a domain $\Omega\subset\mc^n$ is strictly pseudoconvex if it has a $C^2$-boundary defining function $\rho$
    such that $(\rho_{j\bar k}(p))_{1\leq j,k\leq n}$  is positive definite when it restricts to the hyperplane defined by
    $$
    \{(t_1,\cdots,t_n)\in\mc^n|\ \sum_{j=1}^nt_j\frac{\partial\rho}{\partial z_j}(p)=0\}
    $$
    for any $p\in \partial\Omega$.\\
     \indent Let $D\subset\mr^n$ be a domain. For $1\leq p\leq \infty$,
     let $L^p(D)$ denote the space of $L^p$-integrable complex valued functions on $D$, and for any $f\in L^p(D)$,  let $\|f\|_{L^p(D)}$ denote the $L^p$-norm of $f$.
     Similarly, we obtain $L^p(\partial D)$ and $\|f\|_{L^p(\partial D)}$ if $D$ has Lipschitz boundary. Moreover, we set
     $$W^{1,p}(D):=\{f\in L^p(D)|\ \frac{\partial f}{\partial x_j}\in L^p(D)\text{ for any }1\leq j\leq n\},$$
     and let $\tr\colon W^{1,p}(D)\rw L^p(\partial D)$ be the trace operator if $D$ is bounded.\\
     \indent Let $\Omega\subset\mc^n$ be a domain, and let $\varphi\in C^2(\overline\Omega)$. We denote by $H_1:=L^2(\Omega,\varphi)$ the Hilbert space of functions on $\Omega$ which are square integrable with respect to $e^{-\varphi}d\lambda$, and $H_2$ (resp. $H_3$) the Hilbert space of forms of type $(0,1)$ (resp. type $(0,2)$) with coefficients in $H_1$. It is clear that
    $$
    T_\varphi\colon H_1\rw H_2,\ u\rwo \bar{\partial}u
    $$
    and
    $$S_\varphi\colon H_2\rw H_3,\ u\rwo \bar{\partial}u
    $$
    are unbounded densely defined, closed operators. Let $T_\varphi^*$ (resp. $S_\varphi^*$) denote the formal adjoint of $T_\varphi$ (resp. $S_\varphi$), and let $\dom(T_\varphi^*)$
    (resp. $\dom(S_\varphi)$) denote the domain of $T_\varphi^*$ (resp. $S_\varphi$). For any element $\alpha\in H_j$, $\|\alpha\|_j$ denotes the norm of $\alpha$
    in $H_j$. Similarly,  $\langle \alpha,\beta\rangle_j$ denotes the corresponding  inner product of $H_j$ for any $\alpha,\ \beta\in H_j$.
    We also let $C_{(0,1)}^1(\overline{\Omega})$ denote the space of forms of type $(0,1)$ with coefficients in $C^1(\overline{\Omega})$.
    For any $\alpha\in \dom(T_\varphi^*)\cap \dom(S_\varphi)$, the graph norm of $\alpha$ is defined by
    $$
    \left(\|\alpha\|_2^2+\|T_\varphi^*\alpha\|_1^2+\|S_\varphi\alpha\|_3^2\right)^{\frac{1}{2}}.
    $$
    \section{The proof of Theorem \ref{thm:sobolev inequality}}\label{section: sobolev}
      To prove Theorem \ref{thm:sobolev inequality}, we make some preparations.\\
      \indent The following lemma is well known, we also contain a proof of it for completeness.
    \begin{lem}\label{lem:interior integral}
    Let $D\subset\mr^n$ be a bounded domain, and let $0<a\leq n$,  then for any $x\in \mr^n$, we have
    $$\int_{D}|x-y|^{-n+a}d\lambda(y)\leq \int_{B(x,R)}|x-y|^{-n+a}d\lambda(y)=\frac{\omega_n}{a}\left(\frac{n|D|}{\omega_n}\right)^{\frac{a}{n}},
    $$
    where $B(x,R)$ is a ball in $\mr^n$ with center at $x$ and radius $R$ whose volume is equal to $|D|.$
    \end{lem}
    \begin{proof}
     We have
     \begin{align*}
    &\quad \int_{D}|x-y|^{-n+a}d\lambda(y)\\
    &=\int_{D-B(x,R)}|x-y|^{-n+a}d\lambda(y)+\int_{D\cap B(x,R)}|x-y|^{-n+a}d\lambda(y)\\
    &\leq R^{-n+a}|D-B(x,R)|+ \int_{D\cap B(x,R)}|x-y|^{-n+a}d\lambda(y)\\
    &\leq \int_{B(x,R)-D}|x-y|^{-n+a}d\lambda(y)+\int_{D\cap B(x,R)}|x-y|^{-n+a}d\lambda(y)\\
    &=\int_{B(x,R)}|x-y|^{-n+a}d\lambda(y)=\omega_n\int_0^{R}r^{-n+a}r^{n-1}dr=\frac{R^a}{a}\omega_n\\
    &=\frac{\omega_n}{a}\left(\frac{n|D|}{\omega_n}\right)^{\frac{a}{n}}.
     \end{align*}
    \end{proof}
    \indent See \cite{Grisvard} for the following well known lemma.
    \begin{lem}\cite[Lemma 1.5.1.9]{Grisvard}\label{lem:normal vector}
    Let $D\subset\mr^n$ be a bounded domain with a Lipschitz boundary defining function $\rho,$  then there is 
    $F\in C^1(\overline D;\mr^n)$ such that 
    $$F(x)\cdot \frac{\nabla\rho(x)}{|\nabla\rho(x)|}\geq 1$$
     for almost every $x\in\partial D$.
    \end{lem}
    In the following, we always choose an $F$ as in Lemma \ref{lem:normal vector}. If $D$ has a $C^2$-boundary defining function $\rho$, we may choose 
    $$F:=\frac{\nabla\rho}{\inf_{x\in\partial D}|\nabla \rho(x)|}$$.
    \indent We also use the following lemma
    \begin{lem}\label{lem:n integral}
    Let $D\subset\mr^n$ be a bounded domain with a Lipschitz boundary defining function $\rho$, then we have
    $$\int_{\partial D}|\tr(f)|dS\leq K\left(\int_D|\nabla f|d\lambda+\int_{D}|f|d\lambda\right),\ \forall f\in W^{1,1}(D),$$
    where
    $$K:=\sup_{x\in D}|F(x)|+\sup_{x\in D}|\nabla\cdot F(x)|.$$
    \end{lem}
    The above lemma is a slight modification of \cite[Theorem 1.5.1.10]{Grisvard}. Considering \cite[Theorem 1.5.1.10]{Grisvard} needs to assume $p\neq 1$, we also contain a quick proof of it
    \begin{proof}
    Since $C^1(\overline D)$ is dense in $W^{1,1}(D)$ and the trace operator is continuous, we may assume $f\in C^1(\overline D)$. Using Integration-by-Parts for functions in Sobolev spaces, we have
    $$\int_{\partial D}|f|F\cdot \frac{\nabla \rho}{|\nabla\rho|}dS=\int_{D}\nabla\cdot(|f|F)d\lambda=\int_{D}\nabla|f|\cdot Fd\lambda+\int_{D}|f|\nabla\cdot Fd\lambda.$$
    Since $|\nabla|f||\leq |\nabla f|$, then the conclusion follows by Lemma \ref{lem:normal vector}.
    \end{proof}
    If $D$ has a $C^2$-boundary defining function $\rho$, then we have  
    $$K:=\frac{\sup_{x\in D}|\nabla\rho(x)|+\sup_{x\in D}|\Delta\rho(x)|}{\inf_{x\in\partial D}|\nabla\rho(x)|}.$$
    \begin{lem}\label{lem:boundary integral}
    Let $D\subset\mr^n(n\geq 2)$ be a bounded domain with a Lipschitz boundary defining function $\rho$, and let $0<a\leq n-1$, then for any $x\in \mr^n$, we have
    $$\int_{\partial D}|x-y|^{-(n-1)+a}dS(y)\leq \frac{Kn\omega_n}{a}\left(\frac{n|D|}{\omega_n}\right)^{\frac{a}{n}}\left(1+\frac{|D|}{\omega_n}\right),$$
    where
    $$K:=\sup_{x\in D}|F(x)|+\sup_{x\in D}|\nabla\cdot F(x)|.$$
    \end{lem}
    \begin{proof}
    Fix $x\in D$, and for any $y\in\mr^n$, set $u(y):=|x-y|^{-(n-1)+a}$, then we have
    $$\nabla u(y)=(-(n-1)+a)|x-y|^{-n+a}\nabla|x-y|,$$
    so we know
    $$|\nabla u(y)|=(n-1-a)|x-y|^{-n+a}.$$
    By Lemma \ref{lem:interior integral} and Lemma \ref{lem:n integral}, the proof is completed  if we note that
    $$\left(\frac{n|D|}{\omega_n}\right)^{\frac{1}{n}}\leq \frac{n|D|}{\omega_n}+1.$$
    \end{proof}
           Now we give the proof of Theorem \ref{thm:sobolev inequality}. Similar to the above, we always choose $F$ as in 
           Lemma \ref{lem:normal vector}. We assume $n=1,$ the general case is similar (we only need to replace Cauchy's Integral Formula
            by Bochner-Martinelli Formula), and its proof is given in Appendix \ref{appendix}. Let us recall Cauchy's Integral Formula. For any $z\in \Omega,$ we have
            $$
            f(z)=\frac{1}{2\pi i}\left(\int_{\Omega}\frac{\frac{\partial f}{\partial\bar{\xi}}(\xi)}{\xi-z}
            d\xi\wedge d\bar{\xi}+\int_{\partial\Omega}\frac{f(\xi)}{\xi-z}d\xi\right)
            $$
            for all $f\in C^1(\Omega)\cap C^0(\overline{\Omega})$ such that $\frac{\partial f}{\partial\bar{z}}$ has a continuous continuation to $\overline\Omega$
            (for those who are not familiar with such a general version, please see \cite[Theorem 1.9.1]{Henkin} or \cite[Theorem 3.1]{Ada} for a proof).
            It is clear that
            $$|d\xi\wedge d\bar\xi|=2d\lambda(\xi),\ |d\xi|=dS(\xi),$$
            so it suffices to prove
         \begin{lem}\label{lem:basic lemma}
           Let $\Omega\subset \mc$ be a bounded domain with Lipschitz boundary, and let $p,q,r$ satisfy
        $$
         1\leq p<\infty,\ 1\leq q<\infty,\ 1\leq r<\infty,\ q(2-p)<2p,\ q<2r,
        $$
        then
         \begin{itemize}
           \item[(i)] There exists a constant $\delta:=\delta(p,q)>0$ such that
          $$
           \delta\|B_\Omega f\|_{L^q(\Omega)}\leq |\Omega|^{\frac{1}{2}+\frac{1}{q}-\frac{1}{p}}\|f\|_{L^p(\Omega)},\ \forall f\in L^p( \Omega),
           $$
          where
          $$
          B_\Omega f(z):=\int_{\Omega}\frac{f(\xi)}{\xi-z}d\lambda(\xi),\ \forall z\in \Omega.
          $$
           \item[(ii)] There exists a constant $\delta:=\delta(|\Omega|,|\partial\Omega|,K,q,r)>0$ such that
            $$
            \delta\|B_{\partial\Omega} f\|_{L^q(\Omega)}\leq \|f\|_{L^r(\partial\Omega)},\ \forall f\in L^r(\partial\Omega),
            $$
          where
          $$
          B_{\partial\Omega}f(z):=\int_{\partial\Omega}\frac{f(\xi)}{\xi-z}dS(\xi),\ \forall z\in \Omega,$$
          and 
         $$K:=\sup_{z\in \Omega}|F(z)|+\sup_{z\in \Omega}|\nabla\cdot F(z)|.$$
         \end{itemize}
         \end{lem}
        \begin{proof}
             For any $a>0,$ set
             $$
             M(a,\Omega):=\sup_{z\in \overline{\Omega}}\int_{\Omega}|\xi-z|^{-2+a}d\lambda(\xi),
             $$
             $$
             N(a,\partial\Omega):=\sup_{z\in \Omega}\int_{\partial\Omega}|\xi-z|^{-1+a}dS(\xi),
             $$
        then $M(a,\Omega),\ N(a,\partial\Omega)<\infty$ for any $a>0$ by Lemma \ref{lem:interior integral} and Lemma \ref{lem:boundary integral}.\\
        (i) Since $\Omega$ is bounded, then by H\"older's Inequality, we may assume $q\geq p$. We moreover assume $p>1$
        since the case $p=1$ can be treated by taking a limit. Set
            \begin{align*}
            b:&=\frac{1}{2}\left(\max\left\{-\frac{2}{q},-1\right\}+\min\left\{\frac{p-2}{p},0\right\}\right),\\
            &=\left\{\begin{array}{ll}
            -\frac{1}{p} &\text{ if }q\leq 2,\\
            \frac{1}{2}-\frac{1}{p}-\frac{1}{q}&\text{ if }q\geq 2,\ p\leq 2,\\
            -\frac{1}{q}& \text{ if }p\geq 2,
            \end{array}
            \right.
            \end{align*}
            then by assumption, we have
            $$
            0<2+qb\leq 2,\ 0<p-2-pb\leq 2(p-1).$$
           Choose $a,q_0,r_0$ such that
            $$
            qa=p,\ q_0(1-a)=p,\ \frac{1}{q}+\frac{1}{q_0}+\frac{1}{r_0}=1,
            $$
            where if $p=q$, then $q_0=\infty$.
             Fix any $f\in L^p(\Omega)$, then for any $ z\in \Omega,$ we have (use H\"older's Inequality for three functions)
            \begin{align*}
            &\quad |B_\Omega f(z)| \\
            &\leq \int_{\Omega}|f(\xi)||\xi-z|^{-1}d\lambda(\xi)  \\
            &=\int_{\Omega}|f(\xi)|^{a}|\xi-z|^{b}
            |f(\xi)|^{1-a}|\xi-z|^{-1-b}d\lambda(\xi)  \\
            &\leq \left[\int_{\Omega}|f(\xi)|^{p}|\xi-z|^{qb}d\lambda(\xi)\right]^{\frac{1}{q}}
            \left[\int_{\Omega}|f(\eta)|^pd\lambda(\eta)\right]^{\frac{q-p}{qp}} \\
            &\quad\cdot \left[\int_{\Omega}|w-z|^{(-1-b)\frac{p}{p-1}}d\lambda(w)\right]^{\frac{p-1}{p}},
            \end{align*}
             which implies
            \begin{align*}
            &\quad |B_\Omega f(z)|^q \\
            &\leq \|f\|_{L^p(\Omega)}^{q-p}\int_{\Omega}|f(\xi)|^{p}|\xi-z|^{qb}d\lambda(\xi)\left[\int_{\Omega}|\eta-z|^{(-1-b)\frac{p}{p-1}}d\lambda(\eta)
            \right]^{\frac{q(p-1)}{p}}  \\
            &\leq M\left(\frac{p-2-pb}{p-1},\Omega\right)^{\frac{q(p-1)}{p}}\|f\|_{L^p(\Omega)}^{q-p}\int_{\Omega}|f(\xi)|^{p}|\xi-z|^{qb}d\lambda(\xi) ,
            \end{align*}
            so (by Fubini's Theorem)
            \begin{align*}
            &\quad \|B_\Omega f\|_{L^q(\Omega)}^q \\
            &\leq M(2+qb,\Omega)M\left(\frac{p-2-pb}{p-1},\Omega\right)^{\frac{q(p-1)}{p}}\|f\|_{L^p(\Omega)}^{q-p}\|f\|_{L^p(\Omega)}^p  \\
            &=M(2+qb,\Omega)M\left(\frac{p-2-pb}{p-1},\Omega\right)^{\frac{q(p-1)}{p}}\|f\|_{L^p(\Omega)}^{q} .
            \end{align*}
            By Lemma \ref{lem:interior integral} and a simple computation, we know
            \begin{align*}
            &\quad M(2+qb,\Omega)M\left(\frac{p-2-pb}{p-1},\Omega\right)^{\frac{q(p-1)}{p}}\\
            &\leq \frac{\omega_{2}}{2+qb}\left(\frac{(p-1)\omega_2}{p-2-pb}\right)^{\frac{q(p-1)}{p}}\left(\frac{2n|\Omega|}{\omega_{2}}\right)^{1+\frac{q}{2}-\frac{q}{p}}\\
            &\leq \frac{\omega_{2}}{2+qb}\left(\frac{(p-1)\omega_2}{p-2-pb}+1\right)^{q}\left(\frac{2n|\Omega|}{\omega_{2}}\right)^{1+\frac{q}{2}-\frac{q}{p}}.
            \end{align*}
        (ii) Similar to (i), we may assume $q\geq r>1$.
        \begin{align*}
        b:&=\frac{1}{2}\left(\max\left\{-\frac{2}{q},-1\right\}-\frac{1}{r}\right),\\
        &=\left\{\begin{array}{ll}
        -\frac{1}{2}-\frac{1}{2r} &\text{ if }q\leq 2,\\
        -\frac{1}{q}-\frac{1}{2r}&\text{ if }q\geq 2,\\
        \end{array}
        \right.
        \end{align*}
        then we have
        $$0<2+qb\leq 2,\ 0<-1-rb\leq r-1.$$
            Choose $a,q_0,r_0$ such that
            $$
            qa=r,\ q_0(1-a)=r,\ \frac{1}{q}+\frac{1}{q_0}+\frac{1}{r_0}=1,
            $$
            where if $q=r$, then $q_0=\infty$. Fix $f\in L^r(\partial\Omega)$, then for any $ z\in \Omega,$ we have
            \begin{align*}
             &\quad |B_{\partial\Omega}f(z)| \\
             &\leq \int_{\partial\Omega}|f(\xi)| |\xi-z|^{-1}dS(\xi)  \\
             &=\int_{\partial\Omega}|f(\xi)|^a|\xi-z|^b |f(\xi)|^{1-a}|\xi-z|^{-1-b}dS(\xi)  \\
             &\leq \left[\int_{\partial\Omega}|f(\xi)|^{qa}|\xi-z|^{qb}dS(\xi)\right]^{\frac{1}{q}}
              \left[\int_{\partial\Omega}|f(\eta)|^{q_0(1-a)}dS(\eta)\right]^{\frac{1}{q_0}}  \\
             &\quad \cdot\left[\int_{\partial\Omega}|w-z|^{(-1-b)\frac{r}{r-1}}dS(w)\right]^{\frac{r-1}{r}}  \\
             &\leq N\left(\frac{1+rb}{1-r},\partial\Omega\right)^{\frac{r-1}{r}} \left[\int_{\partial\Omega}|f(\xi)|^{r}|\xi-z|^{qb}dS(\xi)\right]^{\frac{1}{q}}  \\
             &\quad\cdot \left[\int_{\partial\Omega}|f(\eta)|^{r}dS(\eta)\right]^{\frac{q-r}{qr}}  ,
            \end{align*}
            so we get
            \begin{align*}
            &\quad |B_{\partial\Omega}f(z)|^{q} \\
            &\leq N\left(\frac{1+rb}{1-r},\partial\Omega\right)^{\frac{q(r-1)}{r}}
            \int_{\partial\Omega}|f(\xi)|^{r}|\xi-z|^{qb}dS(\xi)  \\
            &\quad\cdot  \left[\int_{\partial\Omega}|f(\eta)|^{r}dS(\eta)\right]^{\frac{q-r}{r}},
            \end{align*}
             then
            \begin{align*}
            &\quad \int_{\Omega}|B_{\partial\Omega} f(z)|^{q}d\lambda(z) \\
            &\leq N\left(\frac{1+rb}{1-r},\partial\Omega\right)^{\frac{q(r-1)}{r}} \int_{(z,\xi)\in \Omega\times\partial\Omega}
            |\xi-z|^{qb}d\lambda(z)dS(\xi)  \\
            &\quad\cdot\int_{\partial\Omega}|f(w)|^{r}dS(w)\cdot\left[\int_{\partial\Omega}|f(\eta)|^{r}dS(\eta)\right]^{\frac{q-r}{r}}  \\
            &\leq M(2+qb,\Omega)|\partial\Omega| N\left(\frac{1+rb}{1-r},\partial\Omega\right)^{\frac{q(r-1)}{r}} \|f\|_{L^r(\partial\Omega)}^q.
            \end{align*}
            By Lemma \ref{lem:interior integral}, and Lemma \ref{lem:boundary integral}, we have
                \begin{align*}
            &\quad \frac{2+qb}{\omega_{2}}M(2+qb,\Omega) N\left(\frac{-1-rb}{r-1},\partial\Omega\right)^{\frac{q(r-1)}{r}}\\
            &\leq \left(1+\frac{|\Omega|}{\omega_{2}}\right)^{\frac{q(r-1)}{r}}\left(\frac{r-1}{-1-rb}\right)^{\frac{q(r-1)}{r}}
            \left(2K\omega_{2}\right)^{\frac{q(r-1)}{r}}\left(\frac{2|\Omega|}{\omega_{2}}\right)^{1-\frac{q}{2r}}\\
            &\leq \left(1+\frac{|\Omega|}{\omega_{2}}\right)^q\left(\frac{1-r}{1+rb}+1\right)^{q}
            \left(2K\omega_{2}+1\right)^{q}\left(\frac{2|\Omega|}{\omega_{2}}\right)^{1-\frac{q}{2r}}\\
            &=\left(1+\frac{|\Omega|}{2\pi}\right)^{q}\left(\frac{1-r}{1+rb}+1\right)^{q}
            \left(4K\pi+1\right)^{q}\left(\frac{|\Omega|}{\pi}\right)^{1-\frac{q}{2r}} .
            \end{align*}
            The proof is completed.
       \end{proof}
    \section{The proof of Corollary \ref{thm: general sobolev inequality}}\label{sec:general sobolev inequality}
    In this section, we give the proof of Corollary \ref{thm: general sobolev inequality}. For convenience, we restate it here.
        \begin{cor}[= Corollary \ref{thm: general sobolev inequality}]
        Let $D\subset\mr^n$ be a bounded domain with Lipschitz boundary and let $p,q,r$ satisfy
       $$
         1\leq p<\infty,\ 1\leq q<\infty,\ 1\leq r< \infty,\ q(2n-p)<2np,\ q(2n-1)<2nr,
       $$
        then there is a constant $\delta:=\delta(D,p,q,r)>0$ such that
        $$
        \delta\|f\|_{L^q(D)}\leq \|\nabla f\|_{L^p(D)}+\|f\|_{L^r(\partial D)},\ \forall f\in C^1(\overline D).
        $$
      \end{cor}
    \begin{proof}
    Let $\Omega_z:=D_x+i\mr_y^n\subset \mc_z^n$. Consider the map
    $$\phi\colon \mc_z^n\rw (\mc_w^*)^n,\ (z_1,\cdots,z_n)\rwo (e^{z_1},\cdots,e^{z_n}),$$
    and let $\Omega':=\phi(\Omega)$, then $\Omega'$ is a bounded domain of $(\mc_w^*)^n$ with Lipschitz boundary. For any Lebesgue measurable subset $U\subset (\mc_w^*)^n$, define
    $$\mu(U):=\left|\{z\in \mc^n|\ \phi(z)\in U,\ \im(z_1),\im(z_2),\cdots,\im(z_n)\in [0,2\pi)\}\right|,$$
    then $\mu$ is a measure on $(\mc_w^*)^n$ (such measure is usually called the push-forward measure). Similarly, we may define the push-forward measure $d\sigma$ on $\partial\Omega'$ of $dS$ on $\partial \Omega$.
    Let $d\lambda'(w)$ is the Lebesgue measure of $(\mc_w^*)^n$ and let $dS'(w)$ be the Hausdorff measure of $\partial\Omega'$, then 
    there is a constant $\delta:=\delta(\Omega')>0$ such that 
    $$\delta d\mu\leq d\lambda'\leq \frac{1}{\delta}d\mu,\ dS'\leq \frac{1}{\delta}d\sigma,\ \inf_{w\in\overline{\Omega'}}\inf_{1\leq j\leq n}|w_j|
    >\frac{\delta}{2}.$$
    We extend $f$ to a $C^1$-function on $\mr_x^n$, and we may regard $f$ as a function on $\mc_z^n$ which is independent of the imaginary part of $z$ in $\mc^n_z$.
    It is clear that $f$ induces a $C^1$-function $f'$ on $(\mc_w^*)^n$ via $\phi$, i.e.
    $$f(e^{z_1},\cdots,e^{z_n})=f(\re(z_1),\cdots,\re(z_n)),\ \forall\ (z_1,\cdots,z_n)\in \mc_z^n.$$
    \indent By Theorem \ref{thm:sobolev inequality}, we know there is a constant $\delta':=\delta'(\Omega',p,q,r)>0$ such that
   $$
    \delta'\|f'\|_{L^q(\Omega')}\leq \|\bar\partial f'\|_{L^p(\Omega')}+\|f'\|_{L^r(\partial\Omega')}.
    $$
    It is obvious that
    $$
    \int_{\Omega'}|f'|^qd\lambda'\geq \delta\int_{\Omega'}|f'(w)|^qd\mu(w)=(2\pi)^n\delta\int_{D}|f|^qd\lambda,
    $$
    and
    $$
    \int_{\partial\Omega'}|f'|^rdS'\leq \frac{1}{\delta}\int_{\partial\Omega'}|f'|^rd\sigma\leq\frac{(2\pi)^n}{\delta}\int_{\partial D}|f|^rd\lambda.
    $$
    For any $w\in (\mc_w^*)^n$, we know
    $$\frac{\partial f'}{\partial\bar{w}_j}(w)=\frac{1}{2\bar{w}_j}\frac{\partial f}{\partial x_j}(\ln|w_1|,\cdots,\ln|w_n|),\ \forall 1\leq j\leq n,$$
    so  we get
    $$|\bar\partial f'|\leq \frac{1}{\delta}|\nabla f|.$$
    Similar to the above, we have
    $$
    \int_{\Omega'}|\bar\partial f'|^pd\lambda'\leq \frac{(2\pi)^n}{\delta^{p+1}}\int_{D}|\nabla f|^pd\lambda.
    $$
    Combing the above inequalities, we get what we wanted.
    \end{proof}
    \section{The proof of Theorem \ref{thm:strictly pseudoconvex}}
    \indent To prove Theorem \ref{thm:strictly pseudoconvex}, let us  recall some lemmas.
    For the notations, see Section \ref{sec:notations}. \\
   \indent Assumption of $\Omega,\ \varphi,\ \psi,\ \rho,\ c$ as in Theorem \ref{thm:strictly pseudoconvex}.  \\
    \indent The following Lemma (see {\cite[Formula 2.1.9]{Hor65}} or {\cite[Lemma 2.31]{Ada}}) describes when an element  $\alpha\in C_{(0,1)}^1(\overline{\Omega})$ belongs to $\dom(T_\varphi^*)$.
    \begin{lem}\label{dom}
    Let $\alpha:=\sum_{j=1}^n \alpha_jd\bar{z}_j\in C_{(0,1)}^1(\overline{\Omega})$, then we know $\alpha\in\dom(T_\varphi^*)$ if and only if
   $$
   \sum_{j=1}^n \alpha_j(z)\frac{\partial\rho}{\partial z_j}(z)=0
   $$
    for any $z\in\partial\Omega$.
    \end{lem}
    \indent By the above Lemma, it is obvious that $\dom(T_\varphi^*)$ does not depend on the weight function $\varphi.$ \\
    \indent A basic approximation result proved by H\"ormander is the following
    \begin{lem}\label{approximation}\cite[Proposition 2.1.1]{Hor65}
     $C_{(0,1)}^1(\overline{\Omega})\cap \dom(T_\varphi^*)$ is dense in $\dom(T_\varphi^*)\cap \dom(S_\varphi)$ under the graph norm.
    \end{lem}
    \indent The Kohn-Morrey-H\"ormander Formula (see {\cite[Proposition 2.1.2]{Hor65}} or {\cite[Theorem 1.4.2]{Ber95}}) says that
    \begin{lem}\label{basic identity}
    For any $\alpha:=\sum_{j=1}^n \alpha_jd\bar{z}_j\in C_{(0,1)}^1(\overline{\Omega})\cap \dom(T_\varphi^*)$, we have
    \begin{align*}
    \|T_\varphi^*\alpha\|_{1}^2+\|S_\varphi \alpha\|_3^2&=\sum_{j,k=1}^n\int_{\Omega}\alpha_j\bar{\alpha}_k\frac{\partial^2\varphi}{\partial z_j\partial\bar{z}_k}e^{-\varphi}d\lambda
    +\sum_{j,k=1}^n\int_{\Omega}\left|\frac{\partial \alpha_j}{\partial\bar{z}_k}\right|^2e^{-\varphi}d\lambda \\
    &\quad +\sum_{j,k=1}^n\int_{\partial\Omega}\alpha_j\bar{\alpha}_k\frac{\partial^2\rho}{\partial z_j\partial\bar{z}_k}e^{-\varphi}\frac{dS}{|\nabla\rho|}.
    \end{align*}
    \end{lem}
    We are ready to prove Theorem \ref{thm:strictly pseudoconvex}.
    \begin{thm}[= Theorem \ref{thm:strictly pseudoconvex}]
         Let $\Omega\subset \mc^n$ be a bounded strictly pseudoconvex domain with a $C^2$-boundary defining function
        $\rho,$ and let $c_0$ be the smallest eigenvalue of the complex Hessian $\left(\frac{\partial^2\rho}{\partial z_j\partial\bar{z}_k}\right)$ of $\rho$ on $\partial\Omega$.
        Let $\varphi\in C^2(\overline\Omega)$ be a plurisubharmonic function, and let $\psi\in C^2(\overline\Omega)$ be a strictly plurisubharmonic function.
        Let
        $$K:=\frac{\sup_{z\in\Omega}|\nabla \rho(z)|+\sup_{z\in\Omega}|\Delta\rho(z)|}{\inf_{z\in\partial\Omega}|\nabla \rho(z)|},$$
        and set
        $$m_1:=\inf_{z\in\Omega}e^{-\varphi(z)-\psi(z)},\ m_2:=\sup_{z\in\Omega}e^{-\varphi(z)-\psi(z)},
        \ m_3:=\inf_{z\in \partial\Omega}\frac{c_0(z)}{|\nabla\rho(z)|}.$$
        Then there is a constant $\delta:=\delta(|\Omega|,|\partial\Omega|,m_1,m_2,m_3,n)>0$ such that for any nonzero $\bar{\partial}$-closed Lebesgue measurable $(0,1)$-form $f:=\sum_{j=1}^n f_jd\bar{z}_j$
         satisfying
        $$
        M_f:=\int_{\Omega}\sum_{j,k=1}^n \psi^{j\bar k}f_j\bar{f}_ke^{-\varphi-\psi}d\lambda<\infty,
        $$
        we can solve $\bar{\partial}u=f$ with the estimate
        \begin{equation}\label{formula:L2 estimate}
        \int_{\Omega}|u|^2e^{-\varphi-\psi}d\lambda\leq\frac{\|f\|_{2}}{\sqrt{\|f\|_{2}^2+\delta M_f}}\int_{\Omega}\sum_{j,k=1}^n \psi^{j\bar k}f_j\bar{f}_ke^{-\varphi-\psi}d\lambda,
        \end{equation}
        where $(\psi^{j\bar k})_{1\leq j,k\leq n}$ is the inverse of the complex Hessian of $\psi$,
        $$
        \|f\|_{2}^2:=\sum_{j=1}^n\int_{\Omega}|f_j|^2e^{-\varphi-\psi}d\lambda.$$
    \end{thm}
    \begin{proof}
    We may assume $M_f=1$. Similar to the proof of \cite[Lemma 4.4.1]{Hor}, by the Hahn-Banach extension theorem and the Riesz representation theorem, it suffices to prove that there is a $\delta:=\delta(|\Omega|,|\partial\Omega|,K,m_1,m_2,m_3,n)>0$
    such that for any $\alpha:=\sum_{j=1}^n\alpha_j d\bar{z}_j\in  \dom(T_{\varphi+\psi}^*)\cap \dom(S_{\varphi+\psi})$, we have
    $$
    |\langle \alpha,f\rangle_2|^2+\delta\delta_0\frac{|\langle \alpha,f\rangle_2|^2}{\|f\|_2^2}\leq \|T_{\varphi+\psi}^*\alpha\|_1^2+\|S_{\varphi+\psi}\alpha\|_3^2.
    $$
    By Lemma \ref{approximation}, we may assume $\alpha\in C^1_{(0,1)}(\overline{\Omega})$. By Theorem \ref{thm:poincare inequality with weight}
    , there is a constant $\delta:=\delta(|\Omega|,|\partial\Omega|,K,m_1,m_2,n)>0$ (note the remark after Lemma \ref{lem:n integral}) such that
        $$
        \delta \min\{m_3,1\}\int_{\Omega}|g|^2e^{-\varphi-\psi}d\lambda\leq\sum_{j=1}^n\int_{\Omega}\left|\frac{\partial g}{\partial\bar{z}_j}\right|^2e^{-\varphi-\psi}d\lambda+m_3\int_{\partial\Omega}|g|^2e^{-\varphi-\psi}dS.
        $$
    for all $g\in C^1(\overline\Omega)$.
     By the Cauchy-Schwarz inequality, Lemma \ref{dom} and Lemma \ref{basic identity}, we get
    \begin{align*}
    &\quad |\langle \alpha,f\rangle_2|^2+\delta\min\{m_3,1\}\frac{|\langle \alpha,f\rangle_2|^2}{\|f\|_2^2} \\
    &\leq \sum_{j,k=1}^n\int_{\Omega}\alpha_j\bar{\alpha}_k\psi_{j\bar k}e^{-\varphi-\psi}d\lambda+\delta\min\{m_3,1\}\sum_{j=1}^n\int_{\Omega}|\alpha_j|^2e^{-\varphi-\psi}d\lambda  \\
    &\leq \sum_{j,k=1}^n\int_{\Omega}\alpha_j\bar{\alpha}_k\psi_{j\bar k}e^{-\varphi-\psi}d\lambda+\sum_{j=1}^n\int_{\Omega}\left|\frac{\partial\alpha_j}{\partial\bar{z}_k}\right|^2e^{-\varphi-\psi}d\lambda  \\
    &\quad +\sum_{j=1}^n m_3\int_{\partial\Omega}|\alpha_j|^2e^{-\varphi-\psi}dS  \\
    &\leq \sum_{j,k=1}^n\int_{\Omega}\alpha_j\bar{\alpha}_k\psi_{j\bar k}e^{-\varphi-\psi}d\lambda+\sum_{j=1}^n\int_{\Omega}\left|\frac{\partial\alpha_j}{\partial\bar{z}_k}\right|^2e^{-\varphi-\psi}d\lambda
      \\
    &\quad +\sum_{j,k=1}^n\int_{\partial\Omega}\alpha_j\bar{\alpha}_k\frac{\partial^2\rho}{\partial z_j\partial\bar{z}_k}e^{-\varphi-\psi}\frac{dS}{|\nabla\rho|}  \\
    &\leq \|T_{\varphi+\psi}^*\alpha\|_1^2+\|S_{\varphi+\psi}\alpha\|_3^2,
    \end{align*}
    which completes the proof.
    \end{proof}
    \bibliographystyle{amsplain}
    
    \section*{Appendix}\label{appendix}
    In this appendix, we  give the proof of Lemma \ref{lem:basic lemma} for general $n$.\\
    \indent Let us recall the Bochner-Martinelli formula,  which is a generalization of Cauchy's Integral Formula.
    \begin{lem}[see {\cite[Theorem 1.9.1]{Henkin}} or {\cite[Theorem 3.1]{Ada}}]\label{lem:bochner}
    Let $\Omega\subset \mc^n$ be a bounded domain with Lipschitz boundary, then for any
    $f\in C^1(\Omega)\cap C^0(\overline{\Omega})$ such that $\bar\partial f:=\sum_{j=1}^n\frac{\partial f}{\partial\bar{z}_j}d\bar{z}_j$
    has a continuous continuation to $\overline\Omega$ and any $z\in\Omega$, we have
   $$
   f(z)=\frac{(n-1)!}{(2\pi i)^n}\left(\int_{\partial\Omega}f(\xi)\omega(\xi,z)-\int_{\Omega}\bar{\partial}f(\xi)\wedge \omega(\xi,z)
    \right),
    $$
    where
    $$
    \omega(\xi,z):=\frac{1}{|\xi-z|^{2n}}\sum_{j=1}^n(-1)^{j+1}(\bar{\xi}_j-\bar{z}_j)
    \mathop{\wedge}_{k\neq j}d\bar{\xi}_k\wedge \mathop{\wedge}_{l=1}^nd\xi_l.
    $$
    \end{lem}
     It is clear that
     $$\left|\mathop{\wedge}_{k=1}^nd\bar{\xi}_k\wedge \mathop{\wedge}_{l=1}^nd\xi_l\right|= 2^nd\lambda(\xi),
     \ \sup_{1\leq j\leq n}\left|\mathop{\wedge}_{k\neq j}d\bar{\xi}_k\wedge \mathop{\wedge}_{l=1}^nd\xi_l\right|\leq 2^{n}dS(\xi).$$
     By Lemma \ref{lem:bochner}, we only need to prove the following lemma.
    \begin{lem}\label{lem:basic lemma 2}
    Let $\Omega\subset \mc^n$ be a bounded domain with Lipschitz boundary, and let $p,q,r$ satisfy
        $$
         1\leq p<\infty,\ 1\leq q<\infty,\ 1\leq r<\infty,\ q(2n-p)<2np,\ q(2n-1)<2nr,
        $$
     then
     \begin{itemize}
           \item[(i)] There exists a constant $\delta:=\delta(n,p,q)>0$ such that
          $$
           \delta\|B_\Omega f\|_{L^q(\Omega)}\leq |\Omega|^{\frac{1}{2n}+\frac{1}{q}-\frac{1}{p}}\|f\|_{L^p(\Omega)},\ \forall f\in L^p(\Omega),
           $$
          where
          $$
          B_\Omega f(z):=\int_{\Omega}\frac{(\bar{\xi}_1-\bar{z}_1)f(\xi)}{|\xi-z|^{2n}}d\lambda(\xi),\ \forall z\in \Omega.
          $$
           \item[(ii)] There exists a constant $\delta:=\delta(|\Omega|,|\partial\Omega|,K,n,q,r)>0$ such that
            $$
            \delta\|B_{\partial\Omega} f\|_{L^q(\Omega)}\leq \|f\|_{L^r(\partial\Omega)},\ \forall f\in L^r(\partial \Omega),
            $$
          where
          $$
          B_{\partial\Omega}f(z):=\int_{\partial\Omega}\frac{(\bar{\xi}_1-\bar{z}_1)f(\xi)}{|\xi-z|^{2n}}dS(\xi),\ \forall z\in \Omega,
          $$
          and 
         $$K:=\sup_{z\in \Omega}|F(z)|+\sup_{z\in \Omega}|\nabla\cdot F(z)|.$$
    \end{itemize}
    \end{lem}
     \begin{proof}
       For any $a>0,$ set
         $$
         M(a,\Omega):=\sup_{z\in \overline{\Omega}}\int_{\Omega}|\xi-z|^{-2n+a}d\lambda(\xi),
         $$
         $$
         N(a,\partial\Omega):=\sup_{z\in \Omega}\int_{\partial\Omega}|\xi-z|^{-2n+1+a}dS(\xi),
         $$
         then $M(a,\Omega),\ N(a,\partial\Omega)<\infty$ for any $a>0.$ \\
        (i) We may assume $q\geq p>1$.  Set
        \begin{align*}
        b:&=\frac{1}{2}\left(\max\left\{-\frac{2n}{q},1-2n\right\}+\min\left\{\frac{p-2n}{p},0\right\}\right),\\
        &=\left\{\begin{array}{ll}
        1-n-\frac{n}{p} &\text{ if }q\leq \frac{2n}{2n-1},\\
        \frac{1}{2}-\frac{n}{p}-\frac{n}{q}&\text{ if }q\geq \frac{2n}{2n-1},\ p\leq 2n,\\
        -\frac{n}{q}& \text{ if }p\geq 2n,
        \end{array}
        \right.
        \end{align*}
        then by assumption, we have
        $$
        0<2n+qb\leq 2n,\ 0<p-2n-pb\leq 2n(p-1).
        $$
        Choose $a,q_0,r_0$ such that
        $$
        qa=p,\ q_0(1-a)=p,\ \frac{1}{q}+\frac{1}{q_0}+\frac{1}{r_0}=1,
        $$
        where if $p=q$, then $q_0=\infty$.
         Fix any $f\in C^0(\overline{\Omega})$, then for any $ z\in \Omega,$ we have (use H\"older's Inequality for three functions)
        \begin{align*}
        &\quad |B_\Omega f(z)| \\
        &\leq \int_{\Omega}|f(\xi)||\xi-z|^{-2n+1}d\lambda(\xi)  \\
        &=\int_{\Omega}|f(\xi)|^{a}|\xi-z|^{b}
        |f(\xi)|^{1-a}|\xi-z|^{-2n+1-b}d\lambda(\xi)  \\
        &\leq \left[\int_{\Omega}|f(\xi)|^{p}|\xi-z|^{qb}d\lambda(\xi)\right]^{\frac{1}{q}}
        \left[\int_{\Omega}|f(\eta)|^pd\lambda(\eta)\right]^{\frac{q-p}{qp}}  \\
        &\quad\cdot \left[\int_{\Omega}|w-z|^{(-2n+1-b)\frac{p}{p-1}}d\lambda(w)\right]^{\frac{p-1}{p}},
        \end{align*}
         which implies
        \begin{align*}
        &\quad |B_\Omega f(z)|^q \\
        &\leq \|f\|_{L^p(\Omega)}^{q-p}\int_{\Omega}|f(\xi)|^{p}|\xi-z|^{qb}d\lambda(\xi)\left[\int_{\Omega}|\eta-z|^{(-2n+1-b)\frac{p}{p-1}}d\lambda(\eta)
        \right]^{\frac{q(p-1)}{p}}  \\
        &\leq M\left(\frac{p-2n-pb}{p-1},\Omega\right)^{\frac{q(p-1)}{p}}\|f\|_{L^p(\Omega)}^{q-p}\int_{\Omega}|f(\xi)|^{p}|\xi-z|^{qb}d\lambda(\xi) ,
        \end{align*}
        so (by Fubini's Theorem)
        \begin{align*}
        &\quad \|B_\Omega f\|_{L^q(\Omega)}^q \\
        &\leq M(2n+qb,\Omega)M\left(\frac{p-2n-pb}{p-1},\Omega\right)^{\frac{q(p-1)}{p}}\|f\|_{L^p(\Omega)}^{q-p}\|f\|_{L^p(\Omega)}^p  \\
        &=M(2n+qb,\Omega)M\left(\frac{p-2n-pb}{p-1},\Omega\right)^{\frac{q(p-1)}{p}}\|f\|_{L^p(\Omega)}^{q} .
        \end{align*}
        A direct computation shows that
        \begin{align*}
        &\quad M(2n+qb,\Omega)M\left(\frac{p-2n-pb}{p-1},\Omega\right)^{\frac{q(p-1)}{p}}\\
        &\leq \frac{\omega_{2n}}{2n+qb}\left(\frac{(p-1)\omega_{2n}}{p-2n-pb}\right)^{\frac{q(p-1)}{p}}\left(\frac{2n|\Omega|}{\omega_{2n}}\right)^{1+\frac{q}{2n}-\frac{q}{p}}\\
        &\leq \frac{\omega_{2n}}{2n+qb}\left(\frac{(p-1)\omega_{2n}}{p-2n-pb}+1\right)^{q}\left(\frac{2n|\Omega|}{\omega_{2n}}\right)^{1+\frac{q}{2n}-\frac{q}{p}}.
        \end{align*}
        (ii) Similar to (i), we may assume $q\geq r>1$. Set
        \begin{align*}
        b:&=\frac{1}{2}\left(\max\left\{-\frac{2n}{q},1-2n\right\}+\frac{1-2n}{r}\right),\\
        &=\left\{\begin{array}{ll}
        \left(\frac{1}{2}+\frac{1}{2r}\right)(1-2n) &\text{ if }q< \frac{2n}{2n-1},\\
        -\frac{n}{q}+\frac{1-2n}{2r}&\text{ if }q\geq \frac{2n}{2n-1},\\
        \end{array}
        \right.
        \end{align*}
        then we have
        $$
          0<2n+qb\leq 2n,\ 0<1-2n-rb\leq (2n-1)(r-1).$$
        Choose $a,q_0,r_0$ such that
        $$
        qa=r,\ q_0(1-a)=r,\ \frac{1}{q}+\frac{1}{q_0}+\frac{1}{r_0}=1,
        $$
        where if $q=r$, then $q_0=\infty$.
         Fix  $f\in C^0(\overline{\Omega})$, then for any $ z\in \Omega,$ we have
        \begin{align*}
         &\quad |B_{\partial\Omega}f(z)| \\
         &\leq \int_{\partial\Omega}|f(\xi)| |\xi-z|^{-2n+1}dS(\xi)  \\
         &=\int_{\partial\Omega}|f(\xi)|^a|\xi-z|^b |f(\xi)|^{1-a}|\xi-z|^{-2n+1-b}dS(\xi)  \\
         &\leq \left[\int_{\partial\Omega}|f(\xi)|^{qa}|\xi-z|^{qb}dS(\xi)\right]^{\frac{1}{q}}
          \left[\int_{\partial\Omega}|f(\eta)|^{q_0(1-a)}dS(\eta)\right]^{\frac{1}{q_0}}  \\
         &\quad \cdot\left[\int_{\partial\Omega}|w-z|^{(-2n+1-b)\frac{r}{r-1}}dS(w)\right]^{\frac{r-1}{r}}  \\
         &\leq N\left(\frac{1-2n-rb}{r-1},\partial\Omega\right)^{\frac{r-1}{r}} \left[\int_{\partial\Omega}|f(\xi)|^{r}|\xi-z|^{qb}dS(\xi)\right]^{\frac{1}{q}}  \\
         &\quad\cdot \left[\int_{\partial\Omega}|f(\eta)|^{r}dS(\eta)\right]^{\frac{q-r}{qr}}  ,
        \end{align*}
        so we get
        \begin{align*}
        &\quad |B_{\partial\Omega}f(z)|^{q} \\
        &\leq N\left(\frac{1-2n-rb}{r-1},\partial\Omega\right)^{\frac{q(r-1)}{r}}
        \int_{\partial\Omega}|f(\xi)|^{r}|\xi-z|^{qb}dS(\xi)  \\
        &\quad\cdot  \left[\int_{\partial\Omega}|f(\eta)|^{r}dS(\eta)\right]^{\frac{q-r}{r}},
        \end{align*}
         then
        \begin{align*}
        &\quad \int_{\Omega}|B_{\partial\Omega} f(z)|^{q}d\lambda(z) \\
        &\leq N\left(\frac{1-2n-rb}{r-1},\partial\Omega\right)^{\frac{q(r-1)}{r}} \int_{(z,\xi)\in \Omega\times\partial\Omega}
        |\xi-z|^{qb}d\lambda(z)dS(\xi)  \\
        &\quad\cdot\int_{\partial\Omega}|f(w)|^{r}dS(w)\cdot\left[\int_{\partial\Omega}|f(\eta)|^{r}dS(\eta)\right]^{\frac{q-r}{r}}  \\
        &\leq M(2n+qb,\Omega)|\partial\Omega| N\left(\frac{1-2n-rb}{r-1},\partial\Omega\right)^{\frac{q(r-1)}{r}} \|f\|_{L^r(\partial\Omega)}^q.
        \end{align*}
        A direct computation shows that
        \begin{align*}
        &\quad \frac{2n+qb}{\omega_{2n}}M(2n+qb,\Omega) N\left(\frac{1-2n-rb}{r-1},\partial\Omega\right)^{\frac{q(r-1)}{r}}\\
        &\leq \left(1+\frac{|\Omega|}{\omega_{2n}}\right)^{\frac{q(r-1)}{r}}\left(\frac{r-1}{1-2n-rb}\right)^{\frac{q(r-1)}{r}}
        \left(2nK\omega_{2n}\right)^{\frac{q(r-1)}{r}}\left(\frac{2n|\Omega|}{\omega_{2n}}\right)^{1+\frac{q}{2nr}-\frac{q}{r}}\\
        &\leq \left(1+\frac{|\Omega|}{\omega_{2n}}\right)^q\left(\frac{r-1}{1-2n-rb}+1\right)^{q}
        \left(2nK\omega_{2n}+1\right)^{q}\left(\frac{2n|\Omega|}{\omega_{2n}}\right)^{1+\frac{q}{2nr}-\frac{q}{r}}.
        \end{align*}
         The proof is completed.
    \end{proof}
\end{document}